\documentclass[11pt]{amsart}

\usepackage{amsmath}
\usepackage{amssymb}
\usepackage{graphicx}
\usepackage[hyphens]{url} 
\usepackage[utf8]{inputenc}
\usepackage{mathtools}  
\usepackage{amsthm}
\usepackage{color}
\usepackage{nicefrac}
\usepackage{tcolorbox}

\newtheorem{theorem}{Theorem}[section]
\newtheorem{corollary}[theorem]{Corollary}
\newtheorem{lemma}[theorem]{Lemma}
\newtheorem{proposition}[theorem]{Proposition}
\theoremstyle{definition}
\newtheorem{definition}[theorem]{Definition}

\newtheorem{remark}[theorem]{Remark}
\newtheorem{assumption}[theorem]{Assumption}
\newtheorem{algo}[theorem]{Pricing Scheme}

\numberwithin{equation}{section}

\title[Dynamic pricing under nested logit demand]{Dynamic pricing under nested logit demand}

\author[D. M\"uller]{David M\"uller}

\address[D. M\"uller]{Department of Mathematics, Chemnitz University of Technology, Reichenhainer Str. 41, 09126
Chemnitz, Germany}
\email{{\tt david.mueller@mathematik.tu-chemnitz.de}}

\author[Yu. Nesterov]{Yurii Nesterov}
\address[Yu. Netsterov]{Center for Operations Research and Econometrics (CORE),
Catholic University of Louvain (UCL), 34 voie du Roman
Pays, 1348 Louvain-la-Neuve, Belgium}
\email{\tt yurii.nesterov@uclouvain.be}

\author[V. Shikhman]{Vladimir Shikhman}
\address[V. Shikhman]{Department of Mathematics, Chemnitz University of Technology, Reichenhainer Str. 41, 09126
Chemnitz, Germany}
\email{\tt vladimir.shikhman@mathematik.tu-chemnitz.de}

\keywords{dynamic pricing, discrete choice, nested logit, total expected revenue function, smoothing}

\subjclass[2010]{90C25, 91B24}

\begin{document}

\begin{abstract}
Recently, there is growing interest and need for dynamic pricing algorithms, especially, in the field of online marketplaces by offering smart pricing options for big online stores. We present an approach to adjust prices based on the observed online market data.  The key idea is to characterize optimal prices as minimizers of a total expected revenue function, which turns out to be convex. We assume that consumers face information processing costs, hence, follow a discrete choice demand model, and suppliers are equipped with quantity adjustment costs. We prove the strong smoothness of the total expected revenue function by deriving the strong convexity modulus of its dual. Our gradient-based pricing schemes outbalance supply and demand at the convergence rates of $\mathcal{O}(\frac{1}{t})$ and $\mathcal{O}(\frac{1}{t^2})$, respectively. This suggests that the imperfect behavior of consumers and suppliers helps to stabilize the market.
\end{abstract}

\maketitle


\section{Introduction}
  In the last years, the interest for {\it dynamic pricing} has been growing,  in particular, due to the enormous success of online marketplaces.
%
%
  Proper pricing algorithms become important whenever an intermediary  has to match demand and supply on a market by setting a suitable price. Such situations arise amongst others in financial services or online marketplaces\footnote{ E.\,g., the accommodation sharing platform Airbnb offers the Smart Price option to hosts \cite{Airbnb}. Other online marketplaces are facing similar challenges. Another known and studied application of dynamic pricing is settled in the area of financial services, see e.\,g. \cite{bayesian} and \cite{keskin}, where consumers' demand loans or credits and the bank has to set the price, i.\,e. interest rate for each consumer. The matching of demand and supply for financial products are often done by financial intermediaries.}. 
Classical learning algorithms for dynamic pricing aim to learn the value of each feature. After having learned these values, one could choose the proper price. 

Instead of considering the feature-based learning, we focus in this paper on the notion of equilibrium prices. The equilibrium prices are characterized via a {\it total expected revenue function}, similar to the total excessive revenue function  from \cite{nesterov2017distributed}. By using convex analysis, we describe a way to set prices which clear the market, i.\,e. they balance supply and demand.  Hence, our approach can be settled in the field of dynamic pricing, as it gives a possibility to learn directly from the observed market data. An intermediary constructs a sequence of prices, which converges to the equilibrium. This enables an efficient trade without any deadweight loss at least in the limit. 

As a main novelty in this paper, we introduce market participants with {\it imperfect behavior}. On the demand side, consumers make random errors while choosing between differentiated alternatives. This makes it harder to predict their decisions. Suppliers are assumed to be inflexible, as they face quantity adjustment costs. By our suggested pricing schemes, a pricing agent can find equilibrium prices with a rate of $\mathcal{O}(1/t)$, respectively $\mathcal{O}(1/t^2)$ for an accelerated version. From the economical perspective this is an astonishing result. Compared to the algorithmic equilibrium model without random errors and quantity adjustment costs in \cite{nesterov2017distributed} with the rate $\mathcal{O}(1/\sqrt{t})$, we have an improvement by an order of magnitude. Our main result suggests that the imperfect behavior of market participants contributes to the stability of markets.

 Let us briefly refer to the related approaches for agents' pricing of differentiated goods under discrete choice demand. This has been studied in the context of 
 the {\it oligopolistic price competition}. Here, suppliers maximize their expected profits by setting prices and facing the discrete choice demand of consumers. We recommend \cite{palma} for a basic overview. A central challenge of this approach is to induce analytical properties of the suppliers' profit function. Concavity of the expected revenue under the multinomial logit has been shown in \cite{dong2009dynamic}, where the authors assume identical price sensitivities. This assumption is relaxed in \cite{li2011pricing}. Furthermore, the authors prove concavity for the nested logit expected revenue functions, if the price sensitivities are equal for alternatives within the same nest. The concavity of the expected revenue functions is with respect to the market shares. Note that the multinomial logit expected revenue function is not concave with respect to prices \cite{hanson1996optimizing}.  
 
 The price competition \`a-la Bertrand is modeled classically by a {\it Nash game}. There are several results concerning existence and uniqueness of the corresponding Nash equilibria. For linear random utility models with an $-1/(n+1)$-concave density function and each of the $n$ firms offering exactly one product, a price equilibrium exists \cite{palma}. 
  In the symmetric case where additionally all observable utilities of alternative are equal, the result can be strengthened to a unique Nash equilibrium \cite{palma}. With slightly more assumptions, the same authors show the existence of a subgame perfect Nash equilibrium under nested logit demand. In \cite{li2011pricing} the multiproduct pricing problem is reduced to finding the root of a single valued equation. The equilibrium price of such an oligopolistic competition is determined by using a modified Lambert W function. The idea of dimensionality reduction can be found in \cite{gallego2014multiproduct}, where also the existence of a unique Nash equilibrium is proven. 
  
  Our approach provides a way to efficiently determine equilibrium prices beyond the game-theoretic approach. 
  The market participants are {\it price takers}, and not price setters as in the oligipolistic environment. This allows us to use the {\it convex potential} for the price adjustment. Namely, in order to achieve equilibrium prices, the total expected revenue has to be minimized with respect to prices. { One may wonder: Is the principal who adjusts prices an idealist, purely
instrumental, akin to the Walrasian auctioneer? Although this interpretation is well possible, in this paper the principal is understood as a market operator who is actually authorized and financed by market participants.}  

Our notation is standard. By $\mathbb{R}^n $ we denote the space of n-dimensional vectors, where  the vectors $x = \left(x^{(1)}, \ldots, x^{(n)}\right)^T $ are column vectors.  We write $\mathbb{R}^n_+ $ for the set of vectors with nonnegative components. If not stated otherwise, inequalities for vectors are meant componentwise.
We write $e_{n} \in \mathbb{R}^n$ for the $n$-dimensional vector of all ones.
We introduce the standard inner product in $\mathbb{R}^n$:
\[
\left\langle x,y\right\rangle = \sum\limits_{i=1}^{n} x^{(i)} y^{(i)}. 
\]
 For $x \in \mathbb{R}^n$, we use the norms
\[
\|x\|_1 = \sum\limits_{i=1}^{n} |x^{(i)}|, \quad 
\|x\|_2 = \sqrt{\sum\limits_{i=1}^{n} \left(x^{(i)}\right)^2}, \quad
\|x\|_\infty = \underset{1 \le i \le n}{\max} |x^{(i)}| . 
\] 
For a real number $x$, we denote by $x_+=\max\{x,0\}$ its positive part.
Given a function $f$, we denote its  domain by  $$\mbox{dom} f = \{x \in \mathbb{R}^n \, | \,f(x) < \infty\}.$$ Further, we recall the definition of the convex conjugate of a convex function $f$: \[f^\star(s) = \underset{x \in \mathbb{R}^n}{\sup}\left\langle x,s \right\rangle - f(x), \] where $s$ is a vector of dual variables. Finally, for the $(n-1)$-dimensional simplex we write   \[
\triangle = \left\{q \in \mathbb{R}^n \,\left|\, \sum\limits_{i=1}^{n} q^{(i)} = 1,  q^{(i)} \ge 0, i=1, \ldots, n \right. \right\}. 
\]
 
\section{Discrete choice model}  

 We present the consumer behavior given by additive random utility models. The additive decomposition of utility goes back to psychological experiments accomplished in the 1920's \cite{thurstone}. A formal description of this framework has been first introduced in economic context \cite{gev}, where rational decision-makers choose from a finite set of mutually exclusive alternatives $I=\{1, \ldots, n\}. $ Although the decision rule follows a rational behavior, agents are prone to random errors. The latter describe decision-affecting features which cannot be observable. Each alternative $i \in I$ provides the utility \[v^{(i)} + \epsilon^{(i)}, \] where $ v^{(i)} \in \mathbb{R} $ is the deterministic utility part of the $i$-th alternative and $\epsilon^{(i)} $ is its stochastic error.  We use the following notation for the vectors of deterministic utilities and of random utilities, respectively:
\[
v = \left(v^{(1)}, \ldots, v^{(n)}\right)^T, \quad \epsilon = \left(\epsilon^{(1)}, \ldots, \epsilon^{(n)}\right)^T. 
\] 
As already mentioned, the consumers behave rationally, meaning they maximize utility. Hence, their corresponding surplus is given by the expected maximum utility
\begin{equation}\label{consumer surplus general}
E(v) = \mathbb{E}_\epsilon  \left(\max_{1 \leq i \leq n} v^{(i)} + \epsilon^{(i)}\right).   
\end{equation}
{
Let us briefly give an interpretation of (\ref{consumer surplus general}). Usually, one imagines a researcher who is examining the choice. The consumer's choice depends on many factors, some of them beeing observed by the researcher and some of them not. The utility that the consumer obtains from choosing an alternative is decomposed into a part $v$, that depends on variables that the researcher observes, and a part $\varepsilon$ that depends on variables that the researcher does not observe. In this paper, we follow the {\it rational inattention} interpretation as advocated in \cite{sims2010}. According to the latter, the consumer's choice is intrinsically prone to errors, e.\,g. since the number of goods is too large or they sometimes become tired of the goods' comparison, see also \cite{shum}.}

 Next, we review some important properties of the surplus function E. It is convex and differentiable \cite{palma}. The well-known Williams-Daly-Zachary theorem  states that the gradient of E corresponds to the vector of choice probabilities \cite{gev}, i.\,e. each component gives  the probability that alternative $i $ provides the maximum utility among all alternatives. 
 To see this, let us denote the choice probabilities by 
 \[
  \mathbb{P}^{(i)} = \mathbb{P} \left( v^{(i)} +  \epsilon^{(i)} = \underset{1 \leq i \leq n}{\max} v^{(i)} + \epsilon^{(i)} \right).
 \]
{ Then, the expected maximum utility can be equivalently written as
 \[
    E(v) = \sum_{i=1}^{n} \mathbb{P}^{(i)} \cdot  \mathbb{E}_\epsilon  \left(v^{(i)} + \epsilon^{(i)}\right). 
 \]
 }
 From here we get in terms of partial derivatives of $E$:  \begin{equation}\label{dalytheorem}
\frac{\partial E(v)}{\partial v^{(i)}} = \mathbb{P}^{(i)}.
\end{equation}
{ The formula (\ref{dalytheorem}) holds if we assume that no ties will ever occur in (\ref{consumer surplus general}). In this case, the probability of two alternatives to simultaneously provide the maximum utility becomes zero. 
The latter is, in particular, implied by a stronger assumption widely used in the literature that the random vector $\epsilon $ follows a joint distribution which is absolutely continuous with respect to the Lebesgue measure, see e.\,g. \cite{palma}.} 

Let us specify the discrete choice demand in detail. For our model, we concentrate on random utility errors which follow the nested logit distribution from \cite{gev} given by the probability density function  
\begin{equation}\label{nl density}
f_\epsilon\left(z\right) = \exp\left(-\sum\limits_{\ell=1}^{L} \left(\sum\limits_{i \in N_\ell}^{} e^{-z^{(i)} /\mu_\ell}\right)^{\mu_\ell}\right),
\end{equation}
where $z=\left(z^{(1)},\ldots,z^{(n)}\right)^T \in \mathbb{R}^n$. Here, every alternative $i$ belongs to exactly one nest $N_\ell \subset \{1, \ldots, n\}$ for $\ell = 1, \ldots, L$. Compared to the well-known multinomial logit model with just one nest, the nested logit is more appropriate to model differentiated products. Nested logit allows in particular the violation of the axiom of irrelevance of independent alternatives, see e.\,g. \cite{palma}.  
The consumer surplus \eqref{consumer surplus general} is then \begin{equation}\label{nested logit surplus}
E(v) = \ln\left(\sum\limits_{l=1}^{L} \left(\sum\limits_{i \in N_\ell}^{} e^{v^{(i)} /\mu_\ell}\right)^{\mu_\ell}\right).  
\end{equation} 
The corresponding choice probabilities of an alternative $i \in N_\ell $ can be derived by using \eqref{dalytheorem}, see also \cite{shum}:
\[
\mathbb{P}^{(i)} = \frac{e^{v{(i)}/\mu_\ell} \left(\sum\limits_{j \in N_\ell}^{}e^{v^{(j)}/\mu_\ell}\right)^{\mu_\ell -1}}{\sum\limits_{k=1}^{L} \left(\sum\limits_{j \in N_k}^{} e^{v^{(j)} /\mu_k}\right)^{\mu_k}}.
\] 
Note that the nested logit distribution fulfills the assumption on ties from above.
Equivalently, the choice probabilities can be written as
\[
   \mathbb{P}^{(i)}= 
   \frac{e^{\mu_\ell \ln \sum_{j \in N_\ell} e^{\nicefrac{v^{(j)}}{\mu_\ell}}}}{\displaystyle
     \sum_{k=1}^{L} e^{\mu_k \ln \sum_{j \in N_k} e^{\nicefrac{v^{(j)}}{\mu_k}}}}\cdot   \frac{e^{\nicefrac{v^{(i)}}{\mu_\ell}}}{\displaystyle
     \sum_{j\in N_\ell} e^{\nicefrac{v^{(j)}}{\mu_\ell}}},
\]
where the term 
\[
\mu_\ell \ln \sum_{j \in N_\ell} e^{\nicefrac{v^{(j)}}{\mu_\ell}}
\]
can be interpreted as the inclusive value of the alternatives within the  nest $N_\ell$. 

We comment on the nest specific parameters $\mu_\ell$, $\ell=1, \ldots, L$.
For the sake of completeness the proof of the following Proposition \ref{prop:corr} can be found in Appendix.

\begin{proposition}[Nest parameters as correlations, \cite{ben1973structure}]
\label{prop:corr}
The correlation of errors of different alternatives within the same $\ell$-th nest is $1-\mu_\ell^2$. The errors of alternatives from different nests are uncorrelated.
\end{proposition}

\begin{remark}
	For the analysis of nested logit the condition $0 < \mu_\ell \le 1$ for $\ell=1, \ldots, L$ is usually assumed. Proposition \ref{prop:corr} is in accordance with this. Indeed, the alternatives in the same nest are correlated, while the correlation between the nests vanishes. The nested logit model only allows for nonnegative correlations, i.\,e. $1-\mu_\ell^2 \ge 0$,  $\ell=1, \ldots, L$. The latter is obviously equivalent to $\mu_\ell \le 1$, $\ell=1, \ldots, L$.   \qed
\end{remark}

We are interested in strong smoothness of the surplus function $E$. 
\begin{definition}[Strong smoothness of $E$] 
\label{strong smoothness}
	The surplus function $E: \mathbb{R}^n \to \mathbb{R} $ is $B$-strongly smooth with respect to the maximum norm $\|\cdot\|_\infty $ if for all $v, \bar v \in \mathbb{R}^n $ we have: 
	\[
	\|\nabla E(v) - \nabla E(\bar v)\|_1 \le  B ||v - \bar v||_\infty.
	\]
	The smallest constant $B \ge 0$ with this property is called the modulus of smoothness of $E$. 
\end{definition}

In what follows, we use a conjugate duality relation between strong smoothness of $E$ and strong convexity of its conjugate $E^\star$. Let us recall the definition of a strongly convex function.

\begin{definition}[Strong convexity of $E^\star$] 
\label{strong convexity}
The convex conjugate $E^\star: \triangle \to \mathbb{R}$ of the surplus function is $\beta$-strongly convex with respect to $\|\cdot\|_1 $ norm if for all $q, \bar{q} \in  \triangle$ and $\lambda \in [0,1] $ we have: 
\[
E^\star(\lambda q + (1- \lambda)\bar{q}) \le \lambda E^\star(q) + (1-\lambda) E^\star(\bar{q}) - \frac{\beta}{2}\lambda (1-\lambda) ||q-\bar{q}||_1^2.
\]
The biggest constant $\beta > 0 $ with the above property is called the modulus of strong convexity of $E^\star$.
\end{definition}

The convex conjugate of $E$ is explicitly given in \cite{shum}:
\[
E^\star(q) = \sum\limits_{\ell =1}^{L}\mu_\ell\sum\limits_{i \in N_\ell}^{} q^{(i)} \ln q^{(i)} + \sum\limits_{\ell =1}^{L} (1 - \mu_\ell) \left(\sum\limits_{i \in N_\ell}^{} q^{(i)}\right) \ln \left(\sum\limits_{i \in N_\ell}^{} q^{(i)}\right). 
\] 
It has an interpretation of the generalized entropy.

\begin{lemma}[Strong convexity of $E^\star$]
\label{lem:sc.of.nl.conjugate}
	The modulus of strong convexity of $E^\star$ with respect to $\|\cdot\|_1 $ norm is $\beta = \underset{1\leq \ell \leq L}{\min} \; \mu_\ell$. 
\end{lemma}
\begin{proof}
	We begin by examining the first part of the formula for $E^\star$, which we denote for simplicity by
	\[
	f(q) = \sum\limits_{\ell =1}^{L}\mu_\ell\sum\limits_{i \in N_\ell}^{} q^{(i)} \ln q^{(i)}.
	\]
	Basic calculus gives its Hessian with the entries 
	\[ 
	\nabla_{ii}^2f(q) = \frac{\mu_\ell}{q^{(i)}} \quad \text{for all} \; i \in N_\ell,  \quad \nabla_{ij}^2f(q) = 0 \quad  \text{for all} \; j \neq i. \] Consequently, the Hessian $\nabla^2 f(q) $ is a diagonal matrix. The second order criterion for strong convexity with respect to an arbitrary norm $\|\cdot\|$ is given in \cite{nesterovbook}: 
	\[
	\left\langle \nabla^2f(q)h,h \right\rangle \ge \beta \|h\|^2 \quad \mbox{for all } \; h \in \mathbb{R}^n. 
	\]
Applying this criterion in our case provides
	\begin{align*}
	\left\langle \nabla^2f(q)h,h\right\rangle &= \sum\limits_{\ell =1}^{L}\mu_\ell\sum\limits_{i \in N_\ell}^{} \frac{\left(h^{(i)}\right)^2 }{q^{(i)}} \ge {\beta} \sum\limits_{i=1}^{n} \frac{{\left(h^{(i)}\right)}^2}{q^{(i)}} \\ & \overset{(\star)}{\ge}  {\beta} \left(\sum\limits_{i=1}^n |h^{(i)}|\right)^2 = {\beta} \|h\|_1^2.
\end{align*}
The last inequality $(\star) $ holds due to 
\[
\sum\limits_{i=1}^n |h^{(i)}| = \sum\limits_{i=1}^n \frac{|h^{(i)}|}{\sqrt{q^{(i)}}}\sqrt{q^{(i)}} \le \sqrt{\sum\limits_{i=1}^{n}\frac{(h^{(i)})^2}{q^{(i)}}} \sqrt{\sum\limits_{i=1}^{n} q^{(i)}} = \sqrt{\sum\limits_{i=1}^{n}\frac{(h^{(i)})^2}{q^{(i)}}}.
\] 
Taking squares on both sides of this inequality gives ($\star$).
Overall, $f$ is ${\beta}$-strongly convex. 
Next, we turn our attention to the second part of $E^\star$, denoting the latter by 
\[
g(q)= \sum\limits_{\ell =1}^{L} (1 - \mu_\ell) \left(\sum\limits_{i \in N_\ell}^{} q^{(i)}\right) \ln \left(\sum\limits_{i \in N_\ell}^{} q^{(i)}\right). 
\]
Clearly, $g$ is convex in $q$. 
 It remains to recall that $E^\star$ -- as the sum of a $\beta$-strongly convex function $f$ and the convex function $g$ -- is $\beta$-strongly convex.
\end{proof}

The next result follows immediately. 

\begin{corollary}[Strong smoothness of $E$]
\label{cor:smoothness.nl.surplus}
	The nested logit surplus function $E$ is strongly smooth with modulus $B = \frac{1}{\underset{1 \le \ell \le L}{\min} \mu_\ell}$.
\end{corollary}
\begin{proof}
	We apply  \cite[Theorem 6]{kakade2009duality}. It states that $E^\star $ is $\beta$-strongly convex with respect to the $\|\cdot\|_1 $ norm if and only if $E$ is $\frac{1}{\beta}$-strongly smooth with respect to the dual maximum norm $\|\cdot\|_\infty. $ In the view of Lemma \ref{lem:sc.of.nl.conjugate}, the convex conjugate of $E$ is $ \underset{1 \le \ell \le L}{\min} \mu_\ell$-strongly convex, hence, the assertion follows.
\end{proof}

\begin{remark}[Generalized nested logit]
The nested logit model belongs to a special class of distributions of random errors called generalized nested logit models (GNL), which were introduced in \cite{gnl}. For these models the vector of random errors $\epsilon$ follows the joint distribution 
\[
f_\epsilon\left(z\right) = \frac{\partial^n\exp\left(-\sum\limits_{\ell=1}^{L}\left(\sum\limits_{i=1}^{n}\left(\sigma_{i\ell}\cdot e^{-z^{(i)}}\right)^{1/\mu_\ell}\right)^{\mu_\ell/\mu}\right)}{\partial z^{(1)}\cdots\partial z^{(n)}},
\] 
where $z=\left(z^{(1)},\ldots,z^{(n)}\right)^T \in \mathbb{R}^n$. Different nests $\ell=1,\ldots,L$ are endowed with parameters $\mu_\ell > 0$ reflecting the variance while choosing alternatives within the nests. The variance of the choice among the nests is described by $\mu > 0$. Additionally, $\mu_\ell \leq \mu$ is assumed for all $\ell=1,\ldots,L$. Every alternative can belong to more than one nest, hence, the parameters $\sigma_{i\ell} > 0$ give the share of $i$-th alternative to belong to the $\ell$-th nest. For any fixed $i \in I$ it holds therefore:
\[
\sum_{\ell=1}^{L} \sigma_{i\ell} = 1.
\] 
In the case of nested logit, there is a unique nest $\ell_i \in \{1,\ldots,L\}$ where the $i$-th alternative belongs to, i.\,e. $\sigma_{i\ell_i} = 1$. Thus, the nests are mutually exclusive. Furthermore, we have $\mu = 1$. 
Recently, estimations for the strong smoothness parameter of GNL surplus functions have been derived  in \cite{mueller:2019}:
 \[
   \hat{M} = \frac{2}{\underset{1\le \ell \le L}{\min} \mu_\ell}- 1/\mu.
 \]
 For the nested logit, the estimator
 \[
 \hat{M}=\frac{2}{\underset{1\le \ell \le L}{\min} \mu_\ell} -1 < \frac{2}{\underset{1\le \ell \le L}{\min} \mu_\ell} = 2 B
 \]
 is at most twice bigger than the modulus from Corollary \ref{cor:smoothness.nl.surplus}.
 We note that for other GNL specifications the modulus of strong smoothness is not known yet. \qed
\end{remark}
%
%

\section{Pricing Problem} 

\subsection{Demand}

 In order to face the dynamic pricing, we consider a population of consumers  whose demand follows the nested logit model. 
 We divide the consumers into $J$ representative types and denote the number of consumers corresponding to the type $j $ as $\mathcal{N}_j$. 
 Let $p=\left(p^{(1)}, \ldots, p^{(n)} \right)^T \in \mathbb{R}^n_+$ denote the prices of products $1, \ldots, n$.
 The expected revenue of a consumer of type $j$ is given by
  \[
       E_j(p) = \mathbb{E}_{\epsilon_j} \left(\underset{1 \leq i \leq n}{\max} a^{(i)}_j - p^{(i)} + \epsilon^{(i)}_j\right), 
 \] 
 where $a^{(i)}_j$ and $\epsilon^{(i)}_j$ are the observable and random utility attached to the $i$-th product by a consumer of type $j$, respectively. 
 In other words, the deterministic utility is 
 \[
  v^{(i)}_j = a^{(i)}_j - p^{(i)}.
 \]
 { We point out that utility is taken here as to be transferable}.
 %
%
  Given the price vector $p$, the expected demand of the $i$-th alternative realized by a consumer of type $j$ equals to the choice probability
  \[
  x^{(i)}_j(p) = \mathbb{P} \left( a^{(i)}_j - p^{(i)}+  \epsilon^{(i)}_j = \underset{1 \leq i \leq n}{\max} a^{(i)}_j - p^{(i)}+  \epsilon^{(i)}_j \right).
  \]
  Note that for the demand vector $x_j(p) = \left(x^{(1)}_j, \ldots, x^{(n)}_j \right)^T$ it holds $x_j(p) \in \triangle$. This means that the overall normalized demand is divided between $n$ alternatives according to their  choice probabilities. We refer to $x_j(p)$ as the expected demand of a consumer of type $j$. For the latter it holds due to \eqref{dalytheorem}:
  \[
  x_j(p) = - \nabla E_j(p).
  \]
    We assume that the vector of random utilities $\epsilon_j=\left(\epsilon^{(1)}_j, \ldots, \epsilon^{(n)}_j \right)^T$ follows the nested logit model with nests $N_{j\ell}$, and nest parameters $0<\mu_{j\ell} \leq 1$ for $\ell =1, \ldots, L_j$.

\subsection{Supply}
 Let us start in the general case with $K$ suppliers. Each supplier offers a vector $y_k \in \mathcal{Y}_k$, where $\mathcal{Y}_k \subset \mathbb{R}^n $ is a closed and convex set reflecting the capacity constraints, $k=1,\ldots, K$.
 Each supplier  has a natural supply level $\hat{y}_k \in \mathbb{R}^n$ and faces additional quantity adjustment costs, so that the corresponding cost function is
\[
	c_k\left(y_k\right)= \hat c_k\left(y_k\right) + \Gamma_k\cdot\|y_k-\hat{y}_k\|_2^2,
\]
where $\hat c_k: \mathbb{R}^n \rightarrow \mathbb{R}$ is convex, and $\Gamma_k >0$. Note that $c_k: \mathbb{R}^n \rightarrow \mathbb{R}$ is $\Gamma_k$-strongly convex with respect to $\|\cdot\|_2$. 

\begin{remark}[Quantity adjustment costs]
	We address the issue of quantity adjustment costs. The idea of price rigidity due to adjustment costs is well known in economics, see e.\,g. \cite{mankiw1985small} and \cite{sheshinski1977inflation}. As it is argued in \cite{andersen1995adjustment}, neglecting similar adjustment costs for quantities would cause an asymmetry towards quantity flexibility. Furthermore \cite{ginsburgh1991quantity} provides theoretical justification for modeling costly quantity adjustments.  Additionally, as \cite{danziger2008adjustment} summarizes, there is no empirical evidence for neglecting these costs. In our context, it seems natural to include some sort of adjustment costs on the supply side. While suppliers may be able to react on an increase or decrease of demand, they will have to make short-term adjustments on their plans, e.\,g. they might have to shut down some capacities or must acquire costly new machines, which workers have to be trained for. 
	By the properties of $\|\cdot\|_2$-norm, we penalize deviations from the natural production level in a symmetric way, i.\,e. we assume adjustment costs due to higher demand to be as costly as costs due to a decrease in demand. Beyond that, we assume that suppliers are sensitive towards big deviations.  \qed
\end{remark}

Due to the presence of a pricing agent, the suppliers are price takers. Hence, given the prices $p\in \mathbb{R}^n_+$ of products the $k$-th supplier maximizes the profit
 \begin{equation}
   \label{eq:supply}
    \pi_k(p)=\underset{y_k \in \mathcal{Y}_k}{\max} \left\langle p,y_k \right\rangle - c_k\left(y_k\right). 
\end{equation}
We denote the unique solution of \eqref{eq:supply} by $y_k(p)$. Due to the strong convexity of the cost function $c_k$, the profit $\pi(p)$ is differentiable, and for the supply we have:
\[
 y_k(p) = \nabla \pi(p).
\]
 
\subsection{Market clearing}
  
  In this section, we present a dynamic pricing model which is based on the observed discrete choice demand. Specifically, we derive equilibrium prices assuming the additive random utility behavior of consumers. Our key idea is to characterize a suitable vector of prices which clears the market.  
  As the discrete choice demand is stochastic, we refer to an equilibrium price, whenever it clears the market on average. In other words, the equilibrium price matches total expected
  demand and total supply.
\begin{definition}[Equilibrium price]
\label{def:eq.price}
	A vector $p^\star \in \mathbb{R}^n$ is called equilibrium price, if the market clears on average, i.\,e.
	\[
	  p^\star \geq 0, \quad  \sum\limits_{k=1}^{K} y_k(p^\star) -\sum\limits_{j=1}^{J}\mathcal{N}_j\cdot x_j(p^\star)\geq 0, 
	\]
	and
	\[
	 \left\langle p^\star,  \sum\limits_{k=1}^{K} y_k(p^\star) -\sum\limits_{j=1}^{J}\mathcal{N}_j\cdot x_j(p^\star) \right\rangle = 0.
	\]
\end{definition}
 In order to face the pricing problem, we present a way to characterize such equilibrium prices. For that, we define the total expected revenue function, which is inspired by the total excessive revenue function in \cite{nesterov2017distributed}. The key ingredient is to sum up all the revenues of the market participants, i.\,e. consumers and suppliers, who naturally have different preferences concerning the prices. 
\begin{definition}[Total expected revenue]\label{def:ter}
	The total expected revenue function of the market with discrete choice demand is 
	\begin{equation}\label{eq:ter}
	TER(p) = \sum\limits_{k=1}^{K} 
	\pi_k\left( p\right)+\sum\limits_{j=1}^{J} \mathcal{N}_j E_j(p). 
	 	\end{equation}
\end{definition}
 
 The pricing agent has to outbalance contrary price interests of consumers and suppliers, in order to provide an efficient marketplace and secure participants' loyalty.  Note that the function $TER$ is convex and differentiable. 
 
 In what follows, we show how the pricing agent can take advantage of the total expected revenue function, in order to maximize the participants' welfare. For that, let us characterize equilibrium prices by making an additional assumption.
 \begin{assumption}[Productivity condition]\label{ass:rc}
 	There exist vectors $ \bar{y}_k \in \mathcal{Y}_k$, $k=1, \ldots, K$, and $ \bar{q}_j \in \triangle $, $j=1,\ldots,J$, such that the total supply strictly exceeds the total expected demand:
 	\[
 	\sum\limits_{k=1}^{K}\bar{y}_k > \sum\limits_{j=1}^{J} \mathcal{N}_j\bar{q}_j.
 	\] 
 \end{assumption}
 The productivity condition has an economic justification, namely there must be at least one scenario where a demand can be satisfied by the suppliers. Otherwise, consumers would presumably leave the market, as their demand cannot be matched. Hence, Assumption \ref{ass:rc} is reasonable and not very restrictive.
  
\begin{lemma} \label{lem:sublevel}
 The total expected revenue function $TER$ has bounded sublevel sets. 
 \end{lemma} 
 \begin{proof}
 	From convex duality, we have 
 	\[
 	  E_j(p) =\mathbb{E}\left(\underset{1 \leq i \leq n}{\max} a_j^{(i)} -p^{(i)} + \epsilon^{(i)}_j\right) = \underset{q_j \in \triangle }{\max} \left\langle q_j, a_j-p \right\rangle - E^\star(q_j). 
 	\] 
 	Then for the total revenue function holds:
 	\[
 	\begin{array}{rcl}
 	  TER(p) &=& \sum\limits_{k=1}^{K} \pi_k(p) + \sum\limits_{j=1}^{J} \mathcal{N}_j E_j(p)\\ \\
 	  &=&
 	\sum\limits_{k=1}^{K} \underset{y_k \in \mathcal{Y}_k}{\max} \; \left\langle p,y_k \right\rangle - c_k\left(y_k\right) +\sum\limits_{j=1}^{J} \mathcal{N}_j \underset{q_j \in \triangle }{\max} \left\langle q_j, a_j-p \right\rangle - E^\star(q_j) \\ \\
 	&\ge& \sum\limits_{k=1}^{K} \left\langle p, \bar{y}_k \right\rangle -c_k(\bar{y}_k) + \sum\limits_{j=1}^{J} \mathcal{N}_j \left( \left\langle \bar{q_j}, a_j-p \right\rangle - E^\star(\bar{q_j})\right)\\ \\
 	&= & \left\langle p, \sum\limits_{k=1}^{K}\bar{y}_k - \sum\limits_{j=1}^{J} \mathcal{N}_j\bar{q}_j \right\rangle + \underbrace{ \sum\limits_{j=1}^{J}  \mathcal{N}_j \left(\left\langle \bar{q}_j, a_j \right\rangle - E^\star (\bar{q}_j) \right)- \sum\limits_{k=1}^{K} c_k(\bar{y}_k)}_{= C}.
 	 \\ \\
 	 &=   &\left\langle p, \sum\limits_{k=1}^{K}\bar{y}_k - \sum\limits_{j=1}^{J} \mathcal{N}_j\bar{q}_j \right\rangle + C.
 	\end{array}
 	\]
 	{ Due to Assumption \ref{ass:rc}, there exists $t \in \mathbb{R}_{++} $ such that it holds:
 	\[
 	\left\langle p, \sum\limits_{k=1}^{K}\bar{y}_k - \sum\limits_{j=1}^{J} \mathcal{N}_j\bar{q}_j \right\rangle \ge \left\langle t\cdot e_n, p \right\rangle.
 	\] 
 	Hence, we get for $p\in \mathbb{R}^n_{+}$:
 	\[
 	TER(p) \geq t||p||_1 + C.
 	\]
 	The latter provides that the sublevel sets of $TER$ are bounded.}
 \end{proof}

 We now characterize the equilibrium prices. 
 
 \begin{theorem}[Equilibrium prices and minimizers of $TER$] 
 \label{thm:eq.prices}
 	The minimization problem 
 	\begin{equation}
 	\label{eq:minter}
 	  \min_{p \in \mathbb{R}^n_+} TER(p)  
 	\end{equation}
 	is solvable, and its solutions are exactly the equilibrium prices.
 \end{theorem}
 \begin{proof}
 	Because of the convexity of $TER$ and Lemma \ref{lem:sublevel}, the existence of its minimizers $p^\star \in \mathbb{R}^n_+$ is guaranteed. The optimality condition for \eqref{eq:minter} reads:
 	\[
 	   \left\langle \nabla TER\left(p^\star\right), p-p^\star \right\rangle \geq 0 \quad \mbox{ for all } p \in \mathbb{R}^n_+.
 	\]
 	This is equivalent to
 	\[
 	p^\star \geq 0, \quad \nabla TER\left(p^\star\right) \geq 0, \quad
	 \left\langle p^\star, \nabla TER\left(p^\star\right) \right\rangle = 0.
	 	\]
 	By substituting the gradient 
 	\begin{equation}
 	    \label{char.eq.prices}
 	 	\nabla TER\left(p^\star\right) = \sum\limits_{k=1}^{K} y_k(p^\star) - \sum\limits_{j=1}^{J} \mathcal{N}_j x_j(p^\star),
 	\end{equation}
 	the latter coincides with the market clearing condition for equilibrium prices. 
 \end{proof}
 Theorem \ref{thm:eq.prices} gives a guidance for the pricing agent. By minimizing $TER$, he clears the market on average and, therefore, avoids deadweight loss. 

We derive the modulus of strong smoothness of $TER$.
 
  \begin{theorem}[Strong smoothness of $TER$] \label{thm:smoothness.TER} 
  	The total expected revenue function $TER$ is  $\left(\sum\limits_{j=1}^{J} \frac{\mathcal{N}_j}{\beta_j} + \sum\limits_{k=1}^{K} \frac{1}{\Gamma_k}\right)$-strongly smooth with respect to $\|\cdot\|_2$, where
  	\[
  	\beta_j = \underset{ 1\leq \ell^\leq L_{j}}{\min} \, \mu_{j\ell}.
  	\]
  \end{theorem}
  \begin{proof}
  Recall that 
\[  
  	TER(p) = \sum\limits_{k=1}^{K} 
	\pi_k\left( p\right) +\sum\limits_{j=1}^{J} \mathcal{N}_j E_j(p). 
\]	
  	The nested logit surplus term $E_j$ is $\frac{1}{\beta_j}$-strongly smooth with respect to $\|\cdot\|_\infty$ for $j=1, \ldots, J$, due to Corollary \ref{cor:smoothness.nl.surplus}. Hence it is also at least $\frac{1}{\beta_j}$-strongly smooth with respect to $\|\cdot\|_2$. 
  	The triangle inequality leads to the $\sum\limits_{j=1}^{J}\frac{\mathcal{N}_j}{\beta_j}$-strongly smoothness for the consumers term. Consider the $k$-th suppliers total costs $$c_k\left(y_k\right) = \hat c_k\left(y_k\right) + \Gamma_k\cdot\|y_k-\hat{y}_k\|^2,$$ which is at least $\Gamma_k$-strongly convex with respect to $\|\cdot\|_2$. Strong smoothness of the supply term follows then by similar argumentation, which concludes the proof. 
  \end{proof}
 

\section{Dynamics}
 In this chapter, we describe the strategy of a pricing agent, who aims to find equilibrium prices, in order to clear the market. While the $TER$-function itself must not be known to the pricing agent, the latter can take advantage of the gradient derived in \eqref{char.eq.prices}.  We make an assumption regarding the available information. 
 \begin{assumption}\label{ass:information}
 	At each period, the pricing agent can observe demand and supply at the market. Furthermore, at least one price $p_0$ fulfilling Assumption \ref{ass:rc}  is known.
 \end{assumption}  
 
 In what follows, we justify Assumption \ref{ass:information} by giving some examples of the pricing agent.

\begin{itemize}
    \item  {\bfseries Online marketplaces and intermediaries.}  Some of these marketplaces offer smart price options, which result in a pricing problem \eqref{eq:minter}. The agent receives differentiated alternatives from the suppliers of the platform and the demand from consumers. By choosing such an option, the suppliers automatically become price takers.  The goal of the pricing agent is to make the website as popular as possible, because nowadays operating a popular website is a valuable asset by itself. Therefore, the pricing agent shall outbalance demand and supply, in order to satisfy the participants. Otherwise, the market would become inefficient, as some of the possibly leaving participants could have been matched by proper pricing.   The setting of an online shop can be regarded as a special case of a marketplace with only one supplier offering goods. Price taking behavior can then be explained via different company departments, e.\,g. the online marketing department operates the online store.  Another popular trend of the e-commerce is so-called flash sales. They are widespread at websites offering discount specials and at traveling booking portals. The main point is to offer some fixed amounts of differentiated products, e.\,g.  exclusive holidays, for limited time.  Often the agents offering such a flash sale have the products already bought, hence the problem in Definition \ref{eq:ter} varies to the task of pricing product such that a fixed amount will be sold to an optimal price. 
    \item {\bfseries Financial intermediaries.} A similar scenario arises for the work of financial intermediaries such as brokers. There are potential sellers and buyers of assets, who are willing to make trades. The intermediary works as a market maker and, therefore, has to match supply and demand. Often the broker is paid per trade. Thus, it seems again to be a natural motivation for the broker to enable avoid unmatched demand and supply, i.\,e. to clear the market on average. In the last years, there has been rising popularity of P2P lending marketplaces. On these online platforms, borrowers and lenders are directly brought together. Hence, the P2P platform acts as a kind of intermediary and at the same time an online marketplace, where offering a pricing option is possible.  
\end{itemize}
 
Under Assumption \ref{ass:information} we can define an intuitive update rule for the prices.

\begin{tcolorbox}
	
	\begin{algo}\label{algo:pricing}
		For $t=0,1,2,\ldots$ update 
		\[
		p_{t+1} = \left[p_t - h\cdot \left(\sum\limits_{k=1}^{K}y_k(p_t)-\sum\limits_{j=1}^{J}\mathcal{N}_j\cdot x(p_t)\right)\right]_+,
		\]
		where the stepsize is 
		\[
		h\leq \frac{1}{\left(\sum\limits_{j=1}^{J} \frac{\mathcal{N}_j}{\beta_j} + \sum\limits_{k=1}^{K} \frac{1}{\Gamma_k}\right)}.
		\]
	\end{algo}
\end{tcolorbox}
 	
\noindent
Pricing Scheme \ref{algo:pricing} follows an economically reasonable idea. The agent chooses the new price of each alternative according to the difference between supply and demand. If supply of an alternative in the last period exceeded its demand, then its new  price will be lower than before, and vice versa. As prices have to be nonnegative, the price vector is projected on the nonnegative orthant. The convergence analysis of pricing scheme \ref{algo:pricing} follows  from the analysis of the proximal gradient methods see e.\,g. \cite{beck2017first}. In fact, 
Pricing Scheme \ref{algo:pricing} coincides with the proximal gradient method with constant stepsize for the problem \eqref{eq:minter}. For that, we recall that the function $TER$ is convex. Clearly, the inverse of the strong smoothness parameter of $TER$ is chosen as the largest possible stepsize. The prox-operator for the indicator function of the nonnegative orthant is simply the projection onto $\mathbb{R}^n_+$. The derivation of the gradient of the $TER$-function in \eqref{char.eq.prices} concludes the assertion.
Therefore, the rate of convergence follows from the analysis of the proximal gradient method. 
\begin{theorem}\label{corr:convergence.constant}[e.\,g. \cite{beck2017first}]
	Let Pricing Scheme \ref{algo:pricing} be applied. Then, the sequence $\left(p_t\right)_{t \geq 0}$ converges to an equilibrium price $p^\star$. Moreover it holds for $t \ge 0$:
	\[ TER\left(p_t\right) - TER(p^\star) \le \frac{\|p_0-p^\star\|_2^2}{2th}.
	\]
\end{theorem}
In other words, by applying this strategy, the market is cleared  with an $\mathcal{O}(1/t)$ rate of convergence.
 Although the result of Theorem \ref{corr:convergence.constant} is not surprising from a mathematical point of view, it seems to be unexpected from an economical view. Recently, the study of convergence rates towards market equilibrium has been undertaken. In \cite{nesterov2017distributed}, the convergence rate of order $\mathcal{O}(1/\sqrt{t})$ is shown for a decentralized market with rational participants.  As explained in Section 3, consumers follow a rational behavior in our model, but they are prone to errors. Additionally, we assumed quantity adjustment costs for the supply side. Though at first glance those conditions may seem counterproductive for an efficient market, Theorem \ref{corr:convergence.constant} states that the pricing agent can find an equilibrium price faster than without both restrictions on participants behavior. To clarify the economic idea of the smoothing, we discuss the case of 
 $$h =\frac{1}{\left(\sum\limits_{j=1}^{J} \frac{\mathcal{N}_j}{\beta_j} + \sum\limits_{k=1}^{K} \frac{1}{\Gamma_k}\right)}.$$
 The upper bound in Theorem \ref{corr:convergence.constant} then becomes 
 \begin{equation}\label{eq:upper bound strong smoothness}	
  TER\left(p_t\right) - TER(p^\star) \le \frac{\left({\sum\limits_{j=1}^{J} \frac{\mathcal{N}_j}{\beta_j} + \sum\limits_{k=1}^{K} \frac{1}{\Gamma_k}}\right)\cdot\|p_0-p^\star\|_2^2}{2t}.
 \end{equation}

 Concerning supply side, this acceleration is reasonable because the quantity rigidity additionally hurts the suppliers. Hence, suppliers prefer a stable market, in order to adjust their long term natural supply level. This is also in accord with the upper bound for the precision of $TER$ in Equation \eqref{eq:upper bound strong smoothness}. A bigger parameter $\Gamma_k$ of $k$-th supplier's adjustment cost term will cause a smaller upper bound for the precision of $TER$. 
 We have already mentioned the duality between discrete choice  and rational inattention models. Namely, consumers choosing according to a discrete choice model, can also be seen as facing information processing costs \cite{shum}. Compared to a situation without information processing costs, consumers prefer to have a market with nonvolatile prices, which means lesser information to process. Again, the behavior is reflected in the upper bound in Equation \eqref{eq:upper bound strong smoothness}. { The smoothness parameter of consumers of type $j$ is 
\[
\beta_j = \underset{1\leq \ell \leq L}{\min} \; \mu_{j\ell}
\]
 from Theorem \ref{thm:smoothness.TER}.
 The parameters $\mu_{j\ell}$'s depend on how similar the alternatives within the nests are, see Proposition \ref{prop:corr}. Since the alternatives in the nest with the smallest correlation $1-\mu_{j\ell}^2 \approx 0$ are very different,} i.\,e. $\beta_j$ is close to one, consumers of type $j$ have to pay relatively high information processing costs in this case. Note that then even the alternatives within one nest can provide very different utilities.
 In extreme case $\mu_{j\ell}=1$ and, therefore, each alternative has to be thoroughly taken into account. On the other hand,  $\beta_j$ close to zero indicates that once one of the nests is chosen, the consumer will be very indifferent between the alternatives within. Thus, information processing costs decrease. 
 { Due to Proposition \ref{prop:corr}, 
high correlation of purchase alternatives within the nests corresponds to small $\mu_{j\ell}$'s, and hence, to small $\beta_j$. Consequently, in pricing schemes we perform short steps, since just relatively small step-sizes $h$ can be taken to guarantee the derived convergence rates. Intuitively speaking, the imperfect behaviour of consumers helps to facilitate pricing with respect to the convergence rate.}
 Previous discussion shows that the pricing agent is able to exploit the imperfect behavior of market participants. They will not vary their decisions as much as without quantity adjustment costs and information processing costs. Hence, the pricing agent gains worth information out of every new price set, which leads to faster convergence towards equilibrium prices.



As the nonnegative orthant is a closed and convex subset of $\mathbb{R}^n$ and the function $TER$ is strongly smooth and convex, the pricing agent is able to improve the rate of convergence for clearing the market. The acceleration of first order methods was first introduced in \cite{acceleration}. We suggest the following pricing scheme, which is an application of the fast proximal gradient method presented in \cite{beck2017first}.

\begin{tcolorbox}
\begin{algo}\label{algo:pricing:accel} By setting $q_0=p_0$, $\gamma_0=1$, for $t=0,1,2,\ldots$:
	\begin{itemize}
		\item[1)] Update $$p_{t+1} = \left[p_t - h\cdot \left(\sum\limits_{k=1}^{K}y_k\left(q_t\right)-\sum\limits_{j=1}^{J}\mathcal{N}_j\cdot x\left(q_t\right)\right)\right]_+,$$ 
		where the stepsize is
		\[
		h\leq\frac{1}{\left(\sum\limits_{j=1}^{J} \frac{\mathcal{N}_j}{\beta_j} + \sum\limits_{k=1}^{K} \frac{1}{\Gamma_k}\right)};
		\]
		\item[2)] Set $\gamma_{t+1}= \frac{1+\sqrt{1+4\gamma_t^2}}{2}$;
		\item[3)] Compute $q_{t+1}= p_{t+1} + \left(\frac{\gamma_t-1}{\gamma_{t+1}}\right)\left(p_{t+1}-p_t\right)$.
	\end{itemize}

\end{algo}
\end{tcolorbox}

\noindent
As mentioned above, Pricing Scheme \ref{algo:pricing:accel} is an application of the fast proximal gradient method, for which the $\mathcal{O}(1/t^2)$ rate of convergence has been shown, see e.\,g. \cite{beck2017first}). Without further conditions for the function $TER$, this rate is unimprovable, because the lower bound for first order methods is matched \cite{nesterovbook}.

Finally, we discuss the selection of stepsize parameter $h$.  Pricing schemes \ref{algo:pricing} and \ref{algo:pricing:accel} suggest a constant stepsize, which is less or equal the inverse of the smoothness parameter of $TER$. In practice, however, the exact estimation of the parameters $\beta_{j}$ for $j=1,\ldots,J$ as well as $\Gamma_k$ for $k=1,\ldots,K$ might be a difficult statistical problem. 
Recall that we have:
\[
\beta_j = \underset{1\leq \ell \leq L}{\min} \; \mu_{j\ell},
\]
and, due to the closed form of choice probabilities, the nested logit parameters $\mu_{j\ell}$ can be estimated via maximum likelihood \cite{brownstone1989efficient}. Yet, the exact implementation has to be done carefully, see e.\,g. \cite{heiss2002structural} and \cite{hensher2005applied} for a detailed discussion. 
Estimation of suppliers' adjustment cost parameters $\Gamma_k$  is a matter of current research. 
Better estimation of $\mu_{j\ell}$ for $ \ell=1,\ldots,L_j$ and $j=1,\ldots,J$ as well as $\Gamma_k$ for $k=1,\ldots,K$ leads to a tighter upper bound for the price adjustments.

{
\section{Conclusion}
We conclude that imperfect behaviour of consumers and producers facilitates to iteratively outbalance demand and supply. From the technical point of view, it is based on the property of strong smoothness of the expected maximum utility -- on the consumers' side. Such smoothness stems from the strong convexity of the corresponding conjugate function. As we have shown, this property holds, in particular, for nested logit. It appears that any error distribution for which the conjugate of
the surplus function is strongly convex will do as well. 
To estimate the modulus of smoothness for the surplus function with respect to given general randomness is a matter of current research. Another issue, worth to be mention, concerns the Walrasian auctioneer who updates prices. Previously in \cite{nesterov2017distributed}, we introduced different techniques for price decentralization, such as trade and auction. According to the latter, either producers suggest prices and consumers
choose the lowest, or consumers suggest prices and producers choose the highest. Both strategies successively lead to equilibrium prices. Unfortunately, these price designs introduce nonsmoothness into the total revenue, so that the acceleration in convergence rate up to one order gets lost. We plan to address this obstacle in the next paper. 
}

\section*{Appendix}
For the sake of completeness, we give a  proof of Proposition \ref{prop:corr}.
\begin{proof}
   The choice probability of  any $ i \in N_\ell $ can be written as a product of two logit choice probabilities 
		\begin{equation} \label{choice} 
	\mathbb{P}^{(i)} = \mathbb{P}^{(i|N_\ell)} \cdot \mathbb{P}^{(N_\ell)}. 
	\end{equation}
		The first term in \eqref{choice} denotes the probability to choose alternative $i $ conditional on nest $N_\ell $ has been chosen. This can be regarded as a second stage decision. The remaining term $ \mathbb{P}^{(N_\ell)} $ gives the probability to choose nest $N_\ell $ among all nests, hence the first stage decision.  Equation \eqref{choice} implies independence of the two logits. Hence there must be random variables $\epsilon_{N_\ell}^{(i)}, \; i \in N_\ell, \; \ell=1, \ldots, L $ and $\epsilon^{(N_\ell)}, \; \ell=1, \ldots, L, $ such that the overall utility of every alternative $i \in N_\ell $  can be written as
 
 \begin{equation}\label{utility}
	v^{(i)} = v^{(i)} + \epsilon_{N_\ell}^{(i)} + v^{(N_\ell)} + \epsilon^{(N_\ell)}.  \end{equation}
	Due to \eqref{choice}, the alternative specific error terms $\epsilon_{N_\ell}^{(i)} $ are independent on the nest error terms  $\epsilon^{(N_\ell)}$.  Obviously, the second stage decision only depends on the alternative specific terms, consequently the $\epsilon_{N_\ell}^{(i)}$'s are iid Gumbel distributed with scale parameters $\mu_\ell$, $\ell=1, \ldots, L. $  Since the first stage decision takes a logit form, the nest specific error terms have to follow a distribution such that the random variable $ \underset{i \in N_\ell} {\max} \, v^{(i)} + \epsilon_{N_\ell}^{(i)} + \epsilon^{(N_\ell)} $ is Gumbel with the scale parameter one.
	Independence of the error terms $  \epsilon_{N_\ell}^{(i)} + \epsilon^{(N_\ell)} $ gives 
	\[Var( \epsilon_{N_\ell}^{(i)} + \epsilon^{(N_\ell)}) =Var( \epsilon_{N_\ell}^{(i)}) + Var( \epsilon^{N_\ell}).
	\] 
	Together with 
	\[
	  Var( \epsilon_{N_\ell}^{(i)} + \epsilon^{(N_\ell)}) = \frac{\pi^2}{6},
	\]
	it follows that \[ Var( \epsilon^{N_\ell}) = \frac{\pi^2}{6} - \frac{\mu_\ell^2 \cdot \pi^2}{6}.\] 
	Simple calculation gives for $i,j \in N_\ell: $ \[ Cov(v^{(i)}, v^{(j)}) = Cov(\epsilon^{(N_\ell)}, \epsilon^{(N_\ell)}) = Var(\epsilon^{(N_\ell)}) = \frac{\pi^2}{6}\cdot(1 - \mu_\ell ^2). \] Due to $Var(v^{(i)}) =  Var( \epsilon_{N_\ell}^{(i)}) + Var( \epsilon^{N_\ell}) $, the proposition holds. 
\end{proof}
 

\end{document}